\documentclass[reqno]{amsart}
\usepackage{amsthm, amssymb, amsfonts}
\usepackage{lmodern}
\usepackage{color}
\usepackage{hyperref}
\usepackage[T5]{fontenc}
\usepackage{amscd,amssymb}
\usepackage[v2,cmtip]{xy}
\usepackage{mathrsfs}
\theoremstyle{plain}
\newtheorem{thm}{Theorem}[section]
\newtheorem{prop}[thm]{Proposition}
\newtheorem{lem}[thm]{Lemma}

\newtheorem{corl}[thm]{Corollary}
\theoremstyle{definition}
\newtheorem{defn}[thm]{Definition}

\newtheorem{nota}[thm]{Notation}

\theoremstyle{plain}

\theoremstyle{definition}

\begin{document} 
\title[On the generators of the polynomial algebra]
{}
 \author{\DJ\d{\u a}ng V\~o Ph\'uc$^{\dag}$ and Nguy\~\ecircumflex n Sum$^{\dag,1}$}

\footnotetext[1]{Corresponding author.}

%-------------------------------------------------
%\centerline{\textbf{\Large On the generators of the polynomial algebra}}
%\centerline{\textbf{\Large as a module over the Steenrod algebra}}
\centerline{\textbf{ON THE GENERATORS OF THE POLYNOMIAL ALGEBRA}}
\centerline{\textbf{AS A MODULE OVER THE STEENROD ALGEBRA}}
\bigskip
%\centerline{\textbf{\Large Sur  les g\'en\'erateurs de l'alg\`ebre polynomiale}}
%\centerline{\textbf{\Large comme module sur l'alg\`ebre de Steenrod}}
\centerline{\textbf{SUR  LES G\'EN\'ERATEURS DE L'ALG\`EBRE POLYNOMIALE}}
\centerline{\textbf{COMME MODULE SUR L'ALG\`EBRE DE STENNROD}}
\maketitle

\noindent{\bf Abstract.}
 Let $P_k:= \mathbb F_2[x_1,x_2,\ldots ,x_k]$ be  the polynomial algebra over the prime field of two elements, $\mathbb F_2$, in $k$ variables $x_1, x_2, \ldots , x_k$, each of degree 1.  

We are interested in the {\it  Peterson hit problem} of finding a minimal set of generators for  $P_k$ as a module over the  mod-2 Steenrod algebra, $\mathcal{A}$.  In this paper, we  study the hit problem in degree $(k-1)(2^d-1)$ with $d$ a positive integer. Our result implies the one of Mothebe \cite{mo,mo1}.

\bigskip
\noindent{\bf R\'esum\'e.} Soient $\mathcal A$ l'alg\`ebre de Steenrod mod-2 et $P_k:= \mathbb F_2[x_1,x_2,\ldots ,x_k]$  l'alg\`ebre polynomiale gradu\'ee \`a $k$ g\'en\'erateurs sur le corps \`a deux \'el\'ements $\mathbb F_2$, chaque g\'en\'erateur \'etant de degr\'e 1. 

Nous \'etudions le probl\`eme suivant soulev\'e par F. Peterson: d\'eterminer un syst\`eme minimal de g\'en\'erateurs comme
module sur l'alg\`ebre de Steenrod pour $P_k$, probl\`eme appel\'e {\it hit problem} en anglais.
Dans ce but, nous \'etudions le {\it hit problem} en degr\'e $(k-1)(2^d-1)$ avec $d > 0$. Cette solution implique un r\'esultat de Mothebe ~\cite{mo,mo1}.
\medskip
\medskip
%======================================
\section{Introduction}\label{s1} 
\setcounter{equation}{0}

Let $P_k$ be the graded polynomial algebra $\mathbb F_2[x_1,x_2,\ldots ,x_k]$, with the degree of each $x_i$
being 1. This algebra arises as the cohomology with coefficients in $\mathbb F_2$ of an elementary abelian 2-group of rank $k$.  Then, $P_k$ is a module over the mod-2 Steenrod algebra, $\mathcal A$.  
The action of $\mathcal A$ on $P_k$ is determined by the elementary properties of the Steenrod squares $Sq^i$ and subject to the Cartan formula (see Steenrod and Epstein~\cite{st}).

An element $g$ in $P_k$ is called {\it hit} if it belongs to  $\mathcal{A}^+P_k$, where $\mathcal{A}^+$ is the augmentation ideal of $\mathcal A$. That means $g$ can be written as a finite sum $g = \sum_{u\geqslant 0}Sq^{2^u}(g_u)$ for suitable polynomials $g_u \in P_k$.  

We are interested in the {\it hit problem}, set up by F. Peterson, of finding a minimal set of generators for the polynomial algebra $P_k$ as a module over the Steenrod algebra. In other words, we want to find a basis of the $\mathbb F_2$-vector space $QP_k := P_k/\mathcal A^+P_k = \mathbb F_2 \otimes_{\mathcal A} P_k$. 
%The problem  is an  interesting  and  important  one. It was investigated by many authors.

The hit problem was first studied by Peterson~\cite{pe}, Wood~\cite{wo}, Singer~\cite {si1}, 
and Priddy~\cite{pr}, who showed its relation to several classical problems respectively in cobordism theory, modular representation theory, the Adams spectral sequence for the stable homotopy of spheres, and stable homotopy type of classifying spaces of finite groups. 

The vector space $QP_k$ was explicitly calculated by Peterson~\cite{pe} for $k=1, 2,$ by Kameko~\cite{ka} for $k=3$, and recently by the second author ~\cite{su1,su3} for $k=4$. 
From the results of Wood \cite{wo} and Kameko \cite{ka}, the hit problem is reduced to the case of degree $n$ of the form
\begin{equation} \label{ct1.1}n =  s(2^d-1) + 2^dm,
\end{equation}
where $s, d, m$ are non-negative integers and $1 \leqslant s <k$, (see \cite{su3}.) For $s=k-1$ and $m > 0$, the problem was studied by Crabb and Hubbuck~\cite{ch}, Nam~\cite{na}, Repka and Selick~\cite{res} and the second author ~\cite{su1,su3}.  

In the present paper, we  study the hit problem in degree $n$ of the form (\ref{ct1.1}) with $s=k-1$, $m=0$ and $d$ an arbitrary positive integer.

Denote by $(QP_k)_n$ the subspace of $QP_k$ consisting of the classes represented by the homogeneous polynomials of degree $n$ in $P_k$.  
From the result of Carlisle and Wood \cite{cw} on the boundedness conjecture, one can see that for $d$ big enough, the dimension of $(QP_k)_n$ does not depend on $d$; it depends only on $k$. In this paper, we prove the following.

\medskip\noindent
{\bf Main Theorem.} {\it Let $n=(k-1)(2^d-1)$ with $d$ a positive integer and let $p = \min\{k,d\}$, $q = \min\{k,d-1\}$. If $k \geqslant 3$, then
\begin{equation*}\label{ctc}\dim (QP_k)_n \geqslant c(k,d) := \sum_{t=1}^p\binom kt + (k-3)\binom{k}2\sum_{u=1}^{q}\binom ku ,
\end{equation*}
with equality if and only if either $k=3$ or $k=4,\ d\geqslant 5$ or $k =5,\ d \geqslant 6$.
}
\medskip

Note that $c(k,1) = \binom k1=k$. If $d > k$, then $c(k,d) = \big((k-3)\binom{k}2  + 1\big)(2^k-1)$. At the end of Section \ref{s3}, we show that our result implies Mothebe's result in \cite{mo,mo1}.

 In Section \ref{s2}, we recall  the definition of an admissible monomial in $P_k$ and Singer's  criterion on the hit monomials.  Our results will be presented in Section \ref{s3}.

%====================================
\section{Preliminaries}\label{s2}
\setcounter{equation}{0}

In this section, we recall some needed information from Kameko~\cite{ka} and Singer~\cite{si2}, which will be used in the next section.
\begin{nota} We denote $\mathbb N_k = \{1,2, \ldots , k\}$ and
\begin{align*}
X_{\mathbb J} = X_{\{j_1,j_2,\ldots , j_s\}} =
 \prod_{j\in \mathbb N_k\setminus \mathbb J}x_j , \ \ \mathbb J = \{j_1,j_2,\ldots , j_s\}\subset \mathbb N_k,
\end{align*}
In particular, $X_{\mathbb N_k} =1,\
X_\emptyset = x_1x_2\ldots x_k,$ 
$X_j = x_1\ldots \hat x_j \ldots x_k, \ 1 \leqslant j \leqslant k,$ and $X:=X_k \in P_{k-1}.$

Let $\alpha_i(a)$ denote the $i$-th coefficient  in dyadic expansion of a non-negative integer $a$. That means
$a= \alpha_0(a)2^0+\alpha_1(a)2^1+\alpha_2(a)2^2+ \ldots ,$ for $ \alpha_i(a) =0$ or 1 with $i\geqslant 0$. Set $\alpha(a) = \sum_{i\geqslant 0}\alpha_i(a).$

Let $x=x_1^{a_1}x_2^{a_2}\ldots x_k^{a_k} \in P_k$. Denote $\nu_j(x) = a_j, 1 \leqslant j \leqslant k$.  
Set $\mathbb J_t(x) = \{j \in \mathbb N_k :\alpha_t(\nu_j(x)) =0\},$
for $t\geqslant 0$. Then, we have
$x = \prod_{t\geqslant 0}X_{\mathbb J_t(x)}^{2^t}.$ 
\end{nota}
\begin{defn}
For a monomial  $x$ in $P_k$,  define two sequences associated with $x$ by
\begin{align*} 
\omega(x)=(\omega_1(x),\omega_2(x),\ldots , \omega_i(x), \ldots),\ \
\sigma(x) = (\nu_1(x),\nu_2(x),\ldots ,\nu_k(x)),
\end{align*}
where
$\omega_i(x) = \sum_{1\leqslant j \leqslant k} \alpha_{i-1}(\nu_j(x))= \deg X_{\mathbb J_{i-1}(x)},\ i \geqslant 1.$
The sequence $\omega(x)$ is called  the weight vector of $x$. 

Let $\omega=(\omega_1,\omega_2,\ldots , \omega_i, \ldots)$ be a sequence of non-negative integers.  The sequence $\omega$ is called  the weight vector if $\omega_i = 0$ for $i \gg 0$.
\end{defn}

The sets of the weight vectors and the sigma vectors are given the left lexicographical order. 

  For a  weight vector $\omega$,  we define $\deg \omega = \sum_{i > 0}2^{i-1}\omega_i$.  If  there are $i_0=0, i_1, i_2, \ldots , i_r > 0$ such that $i_1 + i_2 + \ldots + i_r  = m$, $\omega_{i_1+\ldots +i_{s-1} + t} = b_s, 1 \leqslant t \leqslant i_s, 1 \leqslant s \leqslant  r$, and $\omega_i=0$ for all $i > m$, then we write $\omega = (b_1^{(i_1)},b _2^{(i_2)},\ldots , b_r^{(i_r)})$. Denote $b_u^{(1)} = b_u$.  For example, $\omega = (3,3,2,1,1,1,0,\ldots) = (3^{(2)},2,1^{(3)})$.

Denote by   $P_k(\omega)$ the subspace of $P_k$ spanned by monomials $y$ such that
$\deg y = \deg \omega$, $\omega(y) \leqslant \omega$, and by $P_k^-(\omega)$ the subspace of $P_k$ spanned by monomials $y \in P_k(\omega)$  such that $\omega(y) < \omega$. 

\begin{defn}\label{dfn2} Let $\omega$ be a weight vector and $f, g$ two polynomials  of the same degree in $P_k$. 

i) $f \equiv g$ if and only if $f - g \in \mathcal A^+P_k$. If $f \equiv 0$ then $f$ is called hit.

ii) $f \equiv_{\omega} g$ if and only if $f - g \in \mathcal A^+P_k+P_k^-(\omega)$. 
\end{defn}

Obviously, the relations $\equiv$ and $\equiv_{\omega}$ are equivalence ones. Denote by $QP_k(\omega)$ the quotient of $P_k(\omega)$ by the equivalence relation $\equiv_\omega$. Then, we have $QP_k(\omega)= P_k(\omega)/ ((\mathcal A^+P_k\cap P_k(\omega))+P_k^-(\omega))$  and
$(QP_k)_n \cong \bigoplus_{\deg \omega = n}QP_k(\omega)$ (see Walker and Wood \cite{wa1}). 

We note that the weight vector of a monomial is invariant under the permutation of the generators $x_i$, hence $QP_k(\omega)$ has an action of the symmetric group $\Sigma_k$.

For a  polynomial $f \in  P_k(\omega)$, we denote by $[f]_\omega$ the class in $QP_k(\omega)$ represented by $f$. Denote by $|S|$ the cardinal of a set $S$.

\begin{defn}\label{defn3} 
Let $x, y$ be monomials of the same degree in $P_k$. We say that $x <y$ if and only if one of the following holds:  

i) $\omega (x) < \omega(y)$;

ii) $\omega (x) = \omega(y)$ and $\sigma(x) < \sigma(y).$
\end{defn}

\begin{defn}
A monomial $x$ is said to be inadmissible if there exist monomials $y_1,y_2,\ldots, y_m$ such that $y_t<x$ for $t=1,2,\ldots , m$ and $x - \sum_{t=1}^my_t \in \mathcal A^+P_k.$ 

A monomial $x$ is said to be admissible if it is not inadmissible.
\end{defn}

Obviously, the set of the admissible monomials of degree $n$ in $P_k$ is a minimal set of $\mathcal{A}$-generators for $P_k$ in degree $n$. 
Now, we recall a result of Singer \cite{si2} on the hit monomials in $P_k$. 

\begin{defn}\label{spi}  A monomial $z$ in $P_k$   is called a spike if $\nu_j(z)=2^{d_j}-1$ for $d_j$ a non-negative integer and $j=1,2, \ldots , k$. If $z$ is a spike with $d_1>d_2>\ldots >d_{r-1}\geqslant d_r>0$ and $d_j=0$ for $j>r,$ then it is called the minimal spike.
\end{defn}
In \cite{si2}, Singer showed that if $\alpha(n+k) \leqslant k$, then there exists uniquely a minimal spike of degree $n$ in $P_k$.

\begin{lem}\label{bdbs}\

{\rm i)} All the spikes in $P_k$ are admissible and their weight vectors are weakly decreasing. 

{\rm ii)} If a weight vector $\omega$ is weakly decreasing and $\omega_1 \leqslant k$, then there is a spike $z$ in $P_k$ such that $\omega (z) = \omega$. 
\end{lem}

The proof of the this lemma is elementary. The following is a criterion for the hit monomials in $P_k$.

\begin{thm}[See Singer~\cite{si2}]\label{dlsig} Suppose $x \in P_k$ is a monomial of degree $n$, where $\alpha(n + k) \leqslant k$. Let $z$ be the minimal spike of degree $n$. If $\omega(x) < \omega(z)$, then $x$ is hit.
\end{thm}

The following theorem will be used in the next section.
\begin{thm}[See \cite{su1,su3}]\label{dl1} Let $n =\sum_{i=1}^{k-1}(2^{d_i}-1)$ 
with $d_i$ positive integers such that $d_1 > d_2 > \ldots >d_{k-2} \geqslant d_{k-1},$ and let $m = \sum_{i=1}^{ k-2}(2^{d_i-d_{k-1}}-1)$. If $d_{k-1} \geqslant k-1 \geqslant 3$, then
$$\dim (QP_k)_n = (2^k-1)\dim (QP_{k-1})_m.$$
\end{thm}
Note that we correct Theorem 3 in \cite{su1} by replacing the condition  $d_{k-1} \geqslant k-1 \geqslant 1$ with $d_{k-1} \geqslant k-1 \geqslant 3$.

%====================================
\section{Proof of Main Theorem}\label{s3}

Denote
$\mathcal N_k =\big\{(i;I) ; I=(i_1,i_2,\ldots,i_r),1 \leqslant  i < i_1 <  \ldots < i_r\leqslant  k,\ 0\leqslant r <k\big\}.$

\begin{defn} Let $(i;I) \in \mathcal N_k$, let $r = \ell(I)$ be the length of $I$, and let $u$
be an integer with $1 \leqslant  u \leqslant r$. A monomial $x \in P_{k-1}$ is said to be $u$-compatible with $(i;I)$ if all of the following hold:

\smallskip
i) $\nu_{i_1-1}(x)= \nu_{i_2-1}(x)= \ldots = \nu_{i_{(u-1)}-1}(x)=2^{r} - 1$,

ii) $\nu_{i_u-1}(x) > 2^{r} - 1$,

iii) $\alpha_{r-t}(\nu_{i_u-1}(x)) = 1,\ \forall t,\ 1 \leqslant t \leqslant  u$,

iv) $\alpha_{r-t}(\nu_{i_t-1} (x)) = 1,\ \forall t,\ u < t \leqslant r$.
\end{defn}

Clearly, a monomial $x$ can be $u$-compatible with a given $(i;I) \in \mathcal N_k $ for at most one value of $u$. By convention, $x$ is $1$-compatible with $(i;\emptyset)$.

For $1 \leqslant i \leqslant k$, define the homomorphism $f_i: P_{k-1} \to P_k$ of algebras by substituting
$$f_i(x_j) = \begin{cases} x_j, &\text{ if } 1 \leqslant j <i,\\
x_{j+1}, &\text{ if } i \leqslant j <k.
\end{cases}$$

\begin{defn}\label{dfn1} Let $(i;I) \in \mathcal N_k$, $x_{(I,u)} = x_{i_u}^{2^{r-1}+\ldots + 2^{r-u}}\prod_{u< t \leqslant r}x_{i_t}^{2^{r-t}}$ for $r = \ell(I)>0$, $x_{(\emptyset,1)} = 1$.  For a monomial $x$ in $P_{k-1}$, 
we define the monomial  $\phi_{(i;I)}(x)$ in $P_k$ by setting
$$ \phi_{(i;I)}(x) = \begin{cases} (x_i^{2^r-1}f_i(x))/x_{(I,u)}, &\text{if there exists $u$ such that}\\ &\text{$x$ is $u$-compatible with $(i, I)$,}\\
0, &\text{otherwise.}
\end{cases}$$

Then we have an $\mathbb F_2$-linear map $\phi_{(i;I)}:P_{k-1}\to P_k$.
In particular, $\phi_{(i;\emptyset)} = f_i$.
\end{defn}

For a positive integer $b$, denote $\omega_{(k,b)} =((k-1)^{(b)})$ and $\bar \omega_{(k,b)}= ((k-1)^{(b-1)},k-3,1)$. 

\begin{lem}[See \cite{su3}]\label{hq0} Let $b$ be a positive integer and let $j_0, j_1, \ldots , j_{b-1} \in \mathbb N_k$. We set $i = \min\{j_0,\ldots , j_{b-1}\}$, $I = (i_1, \ldots, i_r)$ with $\{i_1, \ldots, i_r\} = \{j_0,\ldots , j_{b-1}\}\setminus \{i\}$. Then, we have 
$$\prod_{0 \leqslant t <b}X_{j_t}^{2^t} \equiv_{\omega_{(k,b)}} \phi_{(i;I)}(X^{2^b-1}).$$
\end{lem}

\begin{defn}
For any $(i;I) \in \mathcal N_k$, we define the homomorphism $p_{(i;I)}: P_k \to P_{k-1}$ of algebras by substituting
$$p_{(i;I)}(x_j) =\begin{cases} x_j, &\text{ if } 1 \leqslant j < i,\\
\sum_{s\in I}x_{s-1}, &\text{ if }  j = i,\\  
x_{j-1},&\text{ if } i< j \leqslant k.
\end{cases}$$
Then, $p_{(i;I)}$ is a homomorphism of $\mathcal A$-modules.  In particular, for $I =\emptyset$,  $p_{(i;\emptyset)}(x_i)= 0$  and $p_{(i;I)}(f_i(y)) = y$ for any $y \in P_{k-1}$. 
\end{defn}

\begin{lem}\label{bdm} If $x$ is a monomial in $P_k$, then $p_{(i;I)}(x) \in P_{k-1}(\omega(x))$. 
\end{lem}

\begin{proof} Set $y = p_{(i;I)}\left(x/x_i^{\nu_i(x)}\right)$. Then, $y$ is a monomial in $P_{k-1}$. If $\nu_i(x) = 0$, then $y = p_{(i;I)}(x)$ and $\omega(y) = \omega(x).$ Suppose $\nu_i(x) >0$ and $\nu_i(x) = 2^{t_1} + \ldots + 2^{t_c}$, where $0 \leqslant t_1 <  \ldots < t_c,\ c\geqslant 1$. 

If $I = \emptyset$, then $p_{(i;I)}(x) = 0$. If $I \ne \emptyset$, then $p_{(i;I)}(x)$ is a sum of monomials of the form 
$\bar y := \big(\prod_{u=1}^cx_{s_u-1}^{2^{t_u}}\big)y$, where $s_u \in I$, $1 \leqslant u \leqslant c$. If $\alpha_{t_u}(\nu_{s_u-1}(y)) =0$ for all $u$, then $\omega(\bar y) = \omega(x)$. Suppose there is an index $u$ such that $\alpha_{t_u}(\nu_{s_u-1}(y)) =1$. Let $u_0$ be the smallest index such that $\alpha_{t_{u_0}}(\nu_{s_{u_0}-1}(y)) =1$. Then, we have
$$ \omega_i(\bar y) =  \begin{cases}\omega _i(x), &\text{if }  i \leqslant t_{u_0},\\ \omega _i(x)-2, &\text{if } i  = t_{u_0}+1.\end{cases}$$
Hence, $\omega (\bar y) < \omega(x)$ and $\bar y \in P_{k-1}(\omega(x))$. 
The lemma is proved.
\end{proof}

Lemma \ref{bdm} implies that if $\omega$ is a weight vector and $x \in P_k(\omega)$, then $p_{(i;I)}(x) \in P_{k-1}(\omega)$.
Moreover, $p_{(i;I)}$ passes to a homomorphism from $QP_k(\omega)$ to $QP_{k-1}(\omega)$. In particular, we have

\begin{lem}[See \cite{su3}]\label{bddl} Let $b$ be a positive integer and let $(j;J), (i;I) \in \mathcal N_k$ with $\ell(I) < b$. 

\medskip
{\rm i)} If $(i;I)\subset (j;J)$, then $p_{(j;J)}\phi_{(i;I)}(X^{2^b-1}) = X^{2^b-1}\ \ \text{\rm mod}(P_{k-1}^-(\omega_{(k,b)})).$

{\rm ii)} If $(i;I)\not\subset (j;J)$, then $p_{(j;J)}\phi_{(i;I)}(X^{2^b-1}) \in P_{k-1}^-(\omega_{(k,b)}).$
\end{lem}

For $0<h\leqslant k$, set $\mathcal N_{k,h} = \{(i;I) \in \mathcal N_k: \ell(I)<h\}$. Then, $|\mathcal N_{k,h}| = \sum_{t=1}^h\binom kt$. 

\begin{prop}\label{mdcm1} Let $d$ be a positive integer and let $p=\min\{k,d\}$.  Then, the set 
$$B(d) :=\big\{\big[\phi_{(i;I)}(X^{2^d-1})\big]_{\omega_{(k,d)}} : (i;I) \in \mathcal N_{k,p}\big\}$$ 
is a basis of the $\mathbb F_2$-vector space $QP_k(\omega_{(k,d)})$. Consequently $\dim QP_k(\omega_{(k,d)}) = \sum_{t=1}^p\binom kt.$
\end{prop}
\begin{proof} Let $x$ be a monomial in $P_k(\omega_{(k,d)})$ and $[x]_{\omega_{(k,d)}} \ne 0$. Then, we have $\omega(x) = \omega_{(k,d)}$. So, there exist $j_0,j_1,\ldots, j_{d-1} \in \mathbb N_k$ such that $x = \prod_{0 \leqslant t < d}X_{j_t}^{2^t}$. According to Lemma \ref{hq0}, there is $(i;I) \in \mathcal N_k$ such that $x = \prod_{0 \leqslant t < d}X_{j_t}^{2^t}\equiv_{\omega_{(k,d)}} \phi_{(i;I)}(X^{2^d-1}),$ where $r = \ell(I) < p = \min\{k,d\}$. Hence, $QP_k(\omega_{(k,d)})$ is spanned by the set $B(d)$. 

Now, we prove that the set $B(d)$ is linearly independent in $QP_k(\omega_{(k,d)})$. Suppose that there is a linear relation
\begin{equation*}\label{ctmdo1}\sum_{(i;I) \in \mathcal N_{k,p}}\gamma_{(i;I)} \phi_{(i;I)}(X^{2^d-1}) \equiv_{\omega_{(k,d)}} 0,\end{equation*}
where $\gamma_{(i;I)} \in \mathbb F_2$. By induction on $\ell(I)$, using Lemma \ref{bdm} and Lemma \ref{bddl} with $b=d$, we can easily show that $\gamma_{(i;I)} =0$ for all $(i;I) \in \mathcal N_{k,p}$. The proposition is proved.
\end{proof}

Set $C_k = \{x_{j_1}x_{j_2}\ldots x_{j_{k-3}}x_j^2: 1\leqslant j_1 < j_2 < \ldots < j_{k-3}<k, \ j_1 \leqslant j <k\} \subset P_{k-1}$.
It is easy to see that $|C_k| = (k-3)\binom{k}2$. 
\begin{lem}\label{bdbbe} $C_k$ is the set of the admissible monomials in $P_{k-1}$ such that their weight vectors are $\bar\omega_{(k,1)}=(k-3,1)$. Consequently, $\dim QP_{k-1}(\bar\omega_{(k,1)}) =  (k-3)\binom{k}2$.
\end{lem}
\begin{proof} Let $z$ be a monomial in $P_{k-1}$ such that $\omega(z) = (k-3,1)$. Then, $z = x_{j_1}x_{j_2}\ldots x_{j_{k-3}}x_j^2$ with $1\leqslant j_1 < j_2 < \ldots < j_{k-3}<k$ and $1 \leqslant j <k$. If $z \not\in C_k$, then $j < j_1$. Then, we have
$$ z = \sum_{s=1}^{k-3}x_{j_s}^2x_{j_1}x_{j_2}\ldots \hat x_{j_s}\ldots x_{j_{k-3}}x_j +Sq^1(x_{j_1}x_{j_2}\ldots x_{j_{k-3}}x_j).$$
Since $x_{j_s}^2x_{j_1}x_{j_2}\ldots \hat x_{j_s}\ldots x_{j_{k-3}}x_j <z$ for $1 \leqslant s \leqslant k-3$, $z$ is inadmissible.

Suppose that $z \in C_k$. If there is an index $s$ such that $j = j_s$, then $z$ is a spike. Hence, by Lemma \ref{bdbs}, it is admissible. Assume that $j \ne j_s$ for all $s$. If $z$ is inadmissible, then there exist monomials $y_1,\ldots, y_m$ in $P_{k-1}$ such that $y_t < z$ for all $t$ and
$z =\sum_{t=1}^m y_t + \sum_{u\geqslant 0} Sq^{2^u}(g_u)$, where $g_u$ are suitable polynomials in $P_{k-1}$. Since $y_t < z$ for all $t$, $z$ is a term of $\sum_{u\geqslant 0} Sq^{2^u}(g_u)$, (recall that a monomial $x$ in $P_k$ is called {\it a term} of a polynomial $f$ if it appears in the expression of $f$ in terms of the monomial basis of $P_k$.) Based on the Cartan formula, we see that $z$ is not a term of $Sq^{2^u}(g_u)$ for all $u>0$. If $z$ is a term of $Sq^{1}(y)$ with $y$ a monomial in $P_{k-1}$, then $y = x_{j_1}x_{j_2}\ldots x_{j_{k-3}}x_j := \tilde y$. So, $\tilde y$ is a term of $g_0$. Then, we have
\begin{align*}\bar y := x_{j_1}^2x_{j_2}\ldots x_{j_{k-3}}x_j = \sum_{s=2}^{k-3}&x_{j_s}^2x_{j_1}x_{j_2}\ldots \hat x_{j_s}\ldots x_{j_{k-3}}x_j\\ &+\sum_{t=1}^m y_t + Sq^1(g_0+\tilde y) + \sum_{u\geqslant 1} Sq^{2^u}(g_u).
\end{align*}
Since $j_1 < j$, we have $y_t < z < \bar y$ for all $t$. Hence, $\bar y$ is a term of $Sq^1(g_0+\tilde y) + \sum_{u\geqslant 1} Sq^{2^u}(g_u)$. By an argument analogous to the previous one, we see that $\tilde y$ is a term of $g_0+\tilde y$. This contradicts the fact that $\tilde y$ is a term of $g_0$. The lemma is proved.
\end{proof}
\begin{prop}\label{mdcm2} Let $d$ be a positive integer and let $q =\min\{k,d-1\}$. Then, the set
$$\bar B(d):=\bigcup_{z \in C_k} \big\{\big[\phi_{(i;I)}(X^{2^{d-1}-1}z^{2^{d-1}})\big]_{\bar\omega_{(k,d)}} :  \ (i;I) \in \mathcal N_{k,q}\big\}$$
is linearly independent in 
$QP_k(\bar\omega_{(k,d)})$. If  $d >k$, then $\bar B(d)$  is a basis of  $QP_k(\bar\omega_{(k,d)})$. 
Consequently $\dim QP_k(\bar\omega_{(k,d)}) \geqslant (k-3)\binom k2\sum_{u=1}^q\binom ku$ with equality if $d>k$.
\end{prop}
\begin{proof} We prove the first part of the proposition. Suppose there is a linear relation
\begin{equation*}\label{ctmdo2}\mathcal S:= \sum_{((i;I),z) \in \mathcal N_{k,q}\times C_k}\gamma_{(i;I),z} \phi_{(i;I)}(X^{2^{d-1}-1}z^{2^{d-1}}) \equiv_{\bar\omega_{(k,d)}} 0,\end{equation*}
where $\gamma_{(i;I),z} \in \mathbb F_2$. We prove $\gamma_{(j;J),z} = 0$ for all $(j;J) \in \mathcal N_{k,q}$ and $z \in C_k$. The proof proceeds by induction on $m=\ell(J)$. Let $(i;I) \in \mathcal N_{k,q}$. Since $r=\ell (I) < q =\min\{k,d-1\}$, $X^{2^{d-1}-1}z^{2^{d-1}}$ is 1-compatible with $(i;I)$ and $x_i^{2^r-1}f_i(X^{2^{d-1}-1})$ is divisible by $x_{(I,1)}$. Hence, using Definition \ref{dfn1}, we easily obtain
\begin{align*}\phi_{(i;I)}(X^{2^{d-1}-1}z^{2^{d-1}}) 
=\phi_{(i;I)}(X^{2^{d-1}-1})f_i(z^{2^{d-1}}). 
\end{align*} 
A simple computation show that if $g \in P_{k-1}^-(\omega_{(k,d-1)})$, then $gz^{2^{d-1}} \in P_{k-1}^-(\bar\omega_{(k,d)})$; if $(i;I) \subset (j;\emptyset)$, then $(i;I) = (j;\emptyset)$; by Lemma \ref{bdm}, $p_{(j;\emptyset)}(\mathcal S) \equiv_{\bar\omega_{(k,d)}} 0$. Hence, applying Lemma  \ref{bddl} with $b = d-1$, we get 
$$p_{(j,\emptyset)}(\mathcal S) \equiv_{\bar\omega_{(k,d)}} \sum_{z\in C_k}\gamma_{(j;\emptyset),z}X^{2^{d-1}-1}z^{2^{d-1}}\equiv_{\bar\omega_{(k,d)}} 0.$$ 
Since $z$ is admissible in $P_{k-1}$, $X^{2^{d-1}-1}z^{2^{d-1}}$ is also admissible in $P_{k-1}$. Hence, the last relation implies $\gamma_{(j;\emptyset),z} = 0$ for all $z \in C_k$. Suppose $0 < m < q$ and $\gamma_{(i;I),z} = 0$ for all $z\in C_k$ and $(i;I) \in \mathcal N_{k,q}$ with $\ell(I) < m$. Let $(j;J) \in \mathcal N_{k,q}$ with $\ell(J) =m$. Note that by Lemma \ref{bdm}, $p_{(j;J)}(\mathcal S) \equiv_{\bar\omega_{(k,d)}} 0$; if $(i;I) \in \mathcal N_{k,q}$, $\ell(I) \geqslant m$ and $(i;I)\subset (j;J)$, then $(i;I) = (j;J)$. So, using Lemma \ref{bddl} with $b=d-1$ and the inductive hypothesis, we obtain
$$p_{(j,J)}(\mathcal S) \equiv_{\bar\omega_{(k,d)}} \sum_{z\in C_k}\gamma_{(j;J),z}X^{2^{d-1}-1}z^{2^{d-1}}\equiv_{\bar\omega_{(k,d)}} 0.$$
From this equality, one gets $\gamma_{(j;J),z} = 0$ for all $z \in C_k$. The first part of the proposition follows.

The proof of the  second part is similar to the one of Proposition 3.3 in \cite{su3}. However, the relation $\equiv_{\bar\omega_{(k,d)}}$ is used in the proof instead of $\equiv$.
\end{proof}
For $k=5$, we have the following result.
\begin{thm}\label{dl12a} Let $n= 4(2^d-1)$ with $d$ a positive integer. The dimension of the $\mathbb F_2$-vector space $(QP_5)_{n}$ is determined by the following table:

\medskip
\centerline{\begin{tabular}{c|ccccc}
$n = 4(2^d-1)$&$d=1$ & $d=2$ & $d=3$&$d=4$  & $d\geqslant 5$\cr
\hline
\ $\dim(QP_5)_n$ & $45$ & $190$ & $480$ &$650$& $651$ \cr
\end{tabular}}
\end{thm} 
Since $n= 4(2^d-1) = 2^{d+1} + 2^d + 2^{d-1} + 2^{d-1} -4$, for $d\geqslant 5$, the theorem follows from Theorem ~ \ref{dl1} and a result in \cite{su3}. For $1 \leqslant d \leqslant 4$, the proof of this theorem is based on Theorem ~ \ref{dlsig} and some results of Kameko \cite{ka}. It is long and very technical.  
The detailed proof of it will be published elsewhere.

\begin{proof}[Proof of Main Theorem] For $k=3$, the theorem follows from the results of Kameko \cite{ka}. For $k=4$, it follows from the results in \cite{su1,su3}. Theorem \ref{dl12a} implies immediately this theorem for $k=5$.

Suppose $k \geqslant 6$.  Lemma \ref{bdbbe} implies that $QP_k(\bar\omega_{(k,1)}) \ne 0$. Hence, 
\begin{align*}\dim (QP_k)_{k-1} &\geqslant  \dim QP_k(\omega_{(k,1)}) +\dim QP_k(\bar\omega_{(k,1)})\\
&> \dim QP_k(\omega_{(k,1)})= k= c(k,1).
\end{align*}
So, the theorem holds for $d=1$. 

Now, let $d>1$ and  $\widetilde\omega_{(k,d)} = ((k-1)^{(d-2)}, k-3,k-4,2)$. Since $\widetilde\omega_{(k,d)}$ is weakly decreasing, by Lemma \ref{bdbs},  $QP_k(\widetilde\omega_{(k,d)}) \ne 0$. We have $\deg(\omega_{(k,d)}) = \deg(\bar\omega_{(k,d)}) = \deg (\widetilde\omega_{(k,d)}) = (k-1)(2^d-1) = n$ and $(QP_k)_n \cong \bigoplus_{\deg \omega = n}QP_k(\omega).$ Hence,  using Propositions \ref{mdcm1} and \ref{mdcm2}, we get
\begin{align*}\dim(QP_k)_n &= \sum_{\deg \omega = n}\dim QP_k(\omega) \\
&\geqslant \dim QP_k(\omega_{(k,d)}) + \dim QP_k(\bar\omega_{(k,d)}) + \dim QP_k(\widetilde\omega_{(k,d)})\\
&> \dim QP_k(\omega_{(k,d)}) + \dim QP_k(\bar\omega_{(k,d)}) \geqslant c(k,d).
\end{align*}
The theorem is proved. 
\end{proof}
 Denote by $N(k,n)$ the number of spikes of degree $n$ in $P_k$.  Note that if $(i;I) \in \mathcal N_k$ and $I \ne \emptyset$, then $\phi_{(i;I)}(x)$ is not a spike for any monomial $x$. Hence, using Propositions \ref{mdcm1} and \ref{mdcm2}, we easily obtain the following.
\begin{corl} Under the hypotheses of Main Theorem,
\begin{align*}\dim(QP_k)_n \geqslant N(k,n) +\sum_{t=2}^p\binom kt + (k-3)\binom{k}2\sum_{u=2}^{q}\binom ku.
\end{align*} 
\end{corl}

This corollary implies Mothebe's result. 

\begin{corl}[See Mothebe \cite{mo,mo1}]\label{hqmo} Under the above hypotheses,
$$\dim (QP_k)_n \geqslant N(k,n)  + \sum_{t=2}^p\binom kt .$$
\end{corl}

\medskip
\noindent
{\bf Acknowledgment.}

We would like to express our warmest thanks to the referee for many helpful suggestions and detailed comments, which have led to improvement of the paper's exposition.

The second author was supported in part by the Research Project Grant No. B2013.28.129 and by the Viet Nam Institute for Advanced Study in Mathematics.

\bigskip
%===============================
{}

\smallskip\noindent
$^{\dag}$\ Department of Mathematics, Quy Nh\ohorn n University, 

\noindent \ \
170 An D\uhorn \ohorn ng V\uhorn \ohorn ng, Quy Nh\ohorn n, B\`inh \DJ\d inh, Viet Nam.

\smallskip\noindent \ \
E-mail: dangphuc150497@gmail.com and nguyensum@qnu.edu.vn

\end{document}